\PassOptionsToPackage{svgnames}{xcolor}
\documentclass[a4paper]{amsart}

\usepackage{breqn}

\usepackage{stmaryrd}
\usepackage[normalem]{ulem}
\usepackage{amsthm}
\usepackage[T1]{fontenc}
\usepackage{lmodern}
\usepackage[utf8]{inputenc}
\usepackage[font=small]{caption}
\usepackage{todonotes}

\usepackage{tikz}
\usepackage[doi=false,isbn=false,backend=biber,style=alphabetic]{biblatex}
\usepackage{microtype}

\usepackage[shortlabels]{enumitem}

\usepackage{thmtools}
\usepackage{hyperref}
\usepackage{cleveref}

\newtheorem{theorem}{Theorem}
\newtheorem*{maintheorem}{Main Theorem}
\newtheorem{conjecture}[theorem]{Conjecture}
\newtheorem*{conjecture*}{Conjecture}
\newtheorem{corollary}[theorem]{Corollary}
\newtheorem*{corollary*}{Corollary}
\newtheorem{proposition}[theorem]{Proposition}
\newtheorem{proposition*}{Proposition}
\newtheorem{lemma}[theorem]{Lemma}
\newtheorem{lemma*}{Lemma}
\newtheorem*{questions*}{Questions}

\theoremstyle{definition}

\newtheorem{definition}[theorem]{Definition}
\newtheorem*{note*}{Note}

\newcommand{\lib}{\llbracket}
\newcommand{\rib}{\rrbracket}

\newcommand{\C}{\mathbb C}
\newcommand{\D}{\mathbb D}

\newcommand{\N}{\mathbb N}
\newcommand{\R}{\mathbb R}
\newcommand{\Q}{\mathbb Q}
\newcommand{\Z}{\mathbb Z}

\newcommand{\RS}{\hat\C}

\newcommand{\ie}{\emph{i.e.\ }}

\newcommand{\eps}{\epsilon}
\newcommand{\tends}{\longrightarrow}

\newcommand{\on}[1]{\operatorname{#1}}
\newcommand{\cal}[1]{\mathcal{#1}}
\newcommand{\ov}[1]{\overline{#1}}
\newcommand{\setof}[2]{\big\{{#1}\,;\,{#2}\big\}}

\newcommand{\dist}{\operatorname{dist}}
\newcommand{\crit}{\operatorname{crit}}

\newcommand{\rad}{\on{rad}}

\newcommand{\loctitle}[1]{\bigskip\noindent\textbf{#1}\medskip}

\title[Big bounded type Siegel disks and critical points]{Biggest bounded type Siegel disks of monic polynomials include those that stick to all critical points}

\author{Xavier Buff}
\author{Arnaud Chéritat}
\author{Pascale Roesch}
\address{Authors affiliation: Univ Toulouse, INSA Toulouse, CNRS, IMT, Toulouse, France.}

\begin{document}

\begin{abstract}
We prove that for all degree $d\geq 2$ and all bounded type irrational $\theta$, in the space of monic polynomials having a period $1$ Siegel disk $\Delta$ of rotation number $\theta$, the maximum locus of the conformal radius of $\Delta$ with respect to its fixed point contains polynomials having all critical points on the boundary of $\Delta$. We apply this to reduce a conjecture of Douady (optimality of the Bruno condition) to a weaker statement. 
\end{abstract}

\maketitle

\tableofcontents


\section{Definitions and main result}

\subsection{Statement}

Given a simply connected strict open subset $U$ of $\C$ and a point $u\in U$, the \emph{conformal radius} of $U$ with respect to $u$ is the unique $r\in(0,+\infty]$ such that there exists a conformal map $\phi:B(0,r)\to U$ with $\phi(0)=u$ and $\phi'(0)=1$.
We denote it $\rad (U,u)$.
We see this as a way to measure the inner size of $U$ as seen from $u$: in particular we recall that $\frac{1}{4}\rad (U,u) \leq d(u,\partial U) \leq \rad (U,u)$.

Let $\theta$ be an irrational real number and $\rho = e^{2\pi i\theta}$. We denote
\[R(z) = R_\theta(z) = e^{2\pi i\theta} z\]
the Euclidean rotation of angle $2\pi \theta$ radians, \ie $\theta$ turns.

Consider a holomorphic map $f$ defined on an open subset of $\C$ and taking values in $\C$.
Assume it has a fixed point $u$ of neutral multiplier.
A \emph{linearization domain} is an open subset $W$ of the domain of $f$, containing $u$ and such that there exists a holomorphic map $\psi$ from $W$ to $\D$ or $\C$ that conjugates $f$ to $R_\theta$ (so $f(W)=W$).
We will denote
\[\rad W = \rad(W,u).\]
If $\theta$ is irrational, then there is a linearization domain that contains all the others, it is called the \emph{Siegel disk} and we denote it $\Delta(f)$.

Fix $d\geq 2$.
For $\theta\in\R$ denote $\cal P$ the set of monic polynomials of degree $d$ and fixing $0$ with multiplier
\[ \rho = e^{2\pi i\theta}
, \]
\ie of the form
\[ P(z)=z^d + a_{d-1} z^{d-1}+ a_{d-3} z^{d-3} + \cdots + a_2 z^2 + \rho z
. \]
This set is parametrized by $a=(a_2,\ldots,a_{d-1})\in\C^{d-2}$ and we sometimes denote $P$ as $P_a$.
Note that if $d=2$, $\cal P$ consists in a single polynomial $P(z) = z^2+\rho z$.

Given $P\in \cal P$, if its fixed point $0$ is linearizable, then we denote $\Delta(P)$, or just $\Delta$, its Siegel disk.
In this case we denote
\[\rad(P)=\rad \Delta(P) = \rad (\Delta(P),0).\]
If $0$ is not linearizable we let $\rad(P)=0$.

We will use the following particular case of a result from \cite{C:these}, the proposition page~67 in Section~II.1, valid for all Bruno numbers $\theta$:
\begin{theorem}[\cite{C:these}]\label{thm:r:cont}
  The function $P\mapsto \rad(P)$ is continuous in $\cal P$.
\end{theorem}

Though we will not use this fact, we find useful to mention that the function $P\mapsto \log \rad(P)$ is pluri-super-harmonic in $\cal P$ (see \Cref{sec:open}).

The following result is attributed to Shishikura and specific to $\theta$ of bounded type (\cite{Shishikura} unpublished).

\begin{theorem}[Shishikura]\label{thm:shi}
  If $\theta$ has bounded type, then there exists $K>1$ such that $\forall P\in \cal P$, $\Delta(P)$ is a $K$-quasicircle containing at least one critical point in its boundary.
\end{theorem}

Uniformity of $K$ is proved in \cite{ZhangGaofei}, and generalized to rational maps, by extending a method of \cite{Zakeri-entire}. We do not know if it was in Shishikura's notes.
That there is a critical point on $\partial \Delta$ also follows from \cite{GS}, independently of the fact that $\partial \Delta$ is a quasicircle.
It follows from \Cref{thm:shi} that $\partial\Delta(P)$ varies continuously with $P$, for the Hausdorff topology on compact sets (see \Cref{lem:Delta:cont}).

A degree $d$ polynomial has exactly $d-1$ critical points counted with multiplicity.
Let us partition $\cal P$ into $d-1$ sets
$ I_k $
according to the number $k\in\{1,\ldots,d-1\}$ of critical points (counted with multiplicity), on the boundary of the Siegel disk.
For $\theta$ of bounded type, since the Siegel disk boundary varies continuously with $P$ and since there is always at least one critical point on $\partial \Delta(P)$, it follows that $I_{1}$ is an open subset of $\cal P$ and $I_{d-1}$ is a closed subset.\footnote{More generally, $I_{\leq k} := \bigcup_{n\leq k} I_{n}$ is open and $I_{\geq k} := \bigcup_{n\geq k} I_{n}$ is closed.}
The set $I_{d-1}$ is contained in the connectivity locus, hence bounded,\footnote{This is a well-known fact. Connected Julia sets of monic polynomials have diameter $\leq 4$, hence the critical points of $P\in I_{d-1}$ lie in $\ov B(0,4)$, from which it follows that $P$, which fixes $0$, has all its coefficient explicitly bounded.} so if $\theta$ has bounded type then it is compact.
In the particular case $d=3$, the set $I_{d-1}$ is known as the \emph{Zakeri arc}, as it was proved by Zakeri to be a Jordan arc in \cite{Zakeri-cubic}.
The work of \cite{ZhangGaofei-TopoPre} implies a generalization to all $d\geq 3$, implying see the note after \Cref{lem:rigid} here.

In the present article, we prove the following result (recall $d\geq 2$ is fixed).

\begin{maintheorem}
  If $\theta$ has bounded type then
  \[ \sup_{P\in\cal P} \rad(P) = \max_{P\in I_{d-1}} \rad(P)
  .\]
\end{maintheorem}

It is tempting to sum this up as follows: \emph{biggest Siegel disks are those for which all critical points are on the boundary} but that would be misleading for a few reasons: though this is likely to be true, we do not claim that a maximizer is necessarily in $I_{d-1}$, only that there is a maximizer in $I_{d-1}$; for this to make sense one has to fix a normalization, otherwise a simple rescaling can make a Siegel disk as big as we want; here we chose to restrict $P$ to $\cal P$, which contains a representative up to linear conjugacy of all polynomials with a fixed point of rotation number $\theta$; the size of the Siegel disk is measured in term of the conformal radius w.r.t.\ the fixed point.

Before detailing the proof, we present a possible application.

\subsection{Application}

We append $[\theta]$ to the objects of the previous section to emphasize their dependence on $\theta$: for instance $\cal P = \cal P[\theta]$.

Bruno's sum was defined in \cite{Bruno} for irrational $\theta$ as $B(\theta)=\sum_{n\geq 0} \frac{\log q_n}{q_{n+1}} \in (0,+\infty]$ where $p_n/q_n$ are the approximants of $\theta$.
A Bruno number is an irrational $\theta$ such that $B(\theta)<+\infty$.
For rational $\theta$, it is useful to let $B(\theta)=+\infty$.
See also \cite{Yoccoz}.

Let us recall a conjecture of Douady (see \cite{Douady-Bourbaki}, Remarque~2 page~162), whose special case $d=2$ was proved to hold by Yoccoz in \cite{Yoccoz}:
\begin{conjecture}[Optimality of the Bruno condition]
  If $\theta\in\R$ is not a Bruno number then no polynomial of degree $\geq 2$ can have a Siegel disk of rotation number $\theta$.
\end{conjecture}
Let us also recall Buff's conjecture (\cite{BC1}, page~328, Conjecture~1.23), where $\crit P$ is the set of critical points of $P$ and $\dist$ is the Euclidean distance in the complex plane:
\begin{conjecture}[\cite{BC1}]\label{conj:buff}
  For every degree $d\geq 2$ there exists $C_d>0$ such that for every Bruno number $\theta$ and for every polynomial $P\in\cal P[\theta]$, 
  \[\rad(P) \leq C_d \, e^{\frac{-B(\theta)}{d-1}} \dist(0,\crit P)
  .\]
\end{conjecture}

\noindent We state here a weaker conjecture, motivated by the proof of \Cref{cor:main} below:
\begin{conjecture}[Reduced conjecture]\label{conj:red}
  For all non-Bruno irrational number $\theta$, denote $[u_0;u_1,u_2,\ldots]$ its continued fraction expansion and let
  $\theta_n = [u_0; u_1, \ldots, u_n, 2,2,2,\ldots]$.
  Denote $M_n = \max \rad(P)$ over $P\in I_{d-1}[\theta_n]$.
  Then
  \[M_n \underset{n\to \infty}{\tends} 0
  .\]
\end{conjecture}

In other words, if all critical points are on the boundary of the Siegel disk of $P\in \cal P[\theta_n]$ then its conformal radius is small.\footnote{Actually, our main theorem proves that the polynomials for which at least one critical point is not on the boundary cannot have a bigger Siegel disks. Hence \Cref{conj:red} implies that actually the supremum of $\rad P$ over $P\in \cal P[\theta_n]$ tends to $0$ as $n\to\infty$.}

Buff's conjecture implies the reduced conjecture: indeed it is easy\footnote{All terms of the Bruno sum are non-negative, $q_n$ is locally constant at irrationals, and at rationals, there is one $n$ such that $q_{n+1}$ tends to infinity while the $q_k$ for $k\leq n$ are locally constant.} to see from the definition of $B$ that $B(\theta_n)\to +\infty$ for \emph{any} sequence $\theta_n$ that tends to a non-Bruno real $\theta$.

\begin{corollary}\label{cor:main}
  Douady's conjecture follows from the reduced conjecture.
\end{corollary}

To prove this corollary we use Lemma~17 from page~145 of \cite{ABC}:
\begin{lemma}[\cite{ABC}]\label{lem:ABC}
  Let $f$ be a holomorphic function that has an indifferent fixed point at $0$, of irrational rotation number $\theta$ (which may or may not be a Bruno number).
  Assume that this fixed point is linearizable and denote $\Delta$ its Siegel disk.
  Denote $[u_0;u_1,u_2,\ldots]$ the continued fraction expansion of $\theta$ and let
  $\theta_n = [u_0; u_1, \ldots, u_n, 2,2,2,\ldots]$, which is a bounded type number.
  Assume that $f_n$ is a sequence of holomorphic maps defined on a disk\footnote{Actually we may take any fixed open set containing $\Delta$.} $B(0,R)$ containing $\Delta$, with $f_n(0)=0$, $f_n'(0)=e^{2\pi i\theta_n}$, and 
  $\lvert f_n-f\rvert \leq K\lvert\theta_n-\theta\rvert$ on $B(0,R)$. 
  Denote $\Delta_n$ the maximal linearization domain of $f_n$.
  Then
  \[ \liminf \rad(\Delta_n) \geq \rad(\Delta)
  .\]
\end{lemma}

Note: The choice of $2,2,2,\ldots$ is somewhat arbitrary and much more flexibility is possible (see \cite{ABC}).

\Cref{lem:ABC} is proved by showing that every compact subset of $\Delta$ is eventually contained in $\Delta_n$. When $\theta$ is a Bruno number this follows from analysis of Yoccoz's sector renormalization methods and estimates of \cite{Yoccoz}, see  Theorem~3 in \cite{ABC-smooth} (this theorem involves a constant $C$ that we take equal to $0$).
For $\theta$ non-Bruno, this required a variation of these method and more technical estimates.

\begin{proof}[Proof of \Cref{cor:main}]
For $\theta\in\R$, denote 
\[P_a[\theta] = z^d + a_{d-1} z^{d-1} + \cdots + a_2 z^2 + e^{2\pi i \theta} z
.\]
Assume by way of contradiction that $\theta$ is a non-Bruno number and that there is a polynomial $P_{a}[\theta] \in\cal P[\theta]$ with a Siegel disk, \ie $\rad(P_a[\theta])>0$.
Let $\theta_n$ be as in \Cref{lem:ABC} and \Cref{conj:red}.
Since $\theta_n\to\theta$, $P_a[\theta_n]\to P_a[\theta]$.
By \Cref{lem:ABC}
\[ \liminf_{n\to\infty} \rad(P_{a}[\theta_n]) \geq \rad(P_{a}[\theta])
.\]
Choose any $r_0\in(0,\rad(P_{a}[\theta])$.
The equation above implies that
\[M'_n := \sup_{P\in \cal P[\theta_n]} \rad(P) \geq r_0
\]
for every $n$ big enough.
Now by the main theorem, we get that $M_n= M'_n$ where $M_n := \max \rad(P)$ over $P\in  I_{d-1}[\theta_n]$. Of course $M_n\leq M'_n$, but from the non-trivial inequality $M_n\geq M'_n$ we get that $M_n \geq r_0$, which contradicts the reduced conjecture.
\end{proof}

To support the reduced conjecture, let us mention that the approach of Dudko and Lyubich in \cite{DL}, which is targeted at quadratic polynomial Siegel disks, is likely to work in the situation where all critical points are in the boundary of the Siegel disk (work in progress by Dzmitry Dudko, Yusheng Luo and Runze Zhang). The authors of \cite{DL} claim that their study allows for a renormalization scheme with enough control so as to precisely know the size of the Siegel disk. In the multi-critical case, there are models under development of what this renormalization does. On the level of conformal radii, it is expected to imply the reduced conjecture (\Cref{conj:red}) with $M_n$ equal to $C_d e^{\frac{-B(\theta_n)}{d-1}}$ (the upper bound in Buff's conjecture \Cref{conj:buff}; we do not include the term $\dist(0,\on{crit} P)$ because it is bounded away from $0$ and $\infty$ on the compact set $I_{d-1}$).

\subsection{Acknowledgments}

We thank Dzmitry Dudko, Yusheng Luo and Runze Zhang for motivating us in proving the main theorem of the present article.
We also thank Carsten Petersen, Saeed Zakeri and Zhang Gaofei for useful discussions about mathematical and historical aspects of Shishikura's method for creating Siegel disks via Blaschke surgery and their adaptations.
We thank Matthieu Astorg and Thomas Gauthier for pointing out to us \Cref{thm:passive} and Romain Dujardin for discussions concerning pluripotential theory, see the open questions in \Cref{sec:open}.
After a first version of this article was posted, we had a discussion with Willie Rush Lim and together we realized that: the work of Zhang Gaofei already proved the rigidity statement of \Cref{thm:Zhang:rigidity}, so we removed the proof from this article and refer to \cite{ZhangGaofei-TopoPre} instead; and that it implies a homeomorphic model of $I_{d-1}$, a fact that does not seem well-known, see the note after \Cref{lem:rigid}.
See also Section~4 of \cite{Lim-Rigidity}.


\section{Proof of the main theorem}\label{sec:pf}

\subsection{Preliminary results}

We will refer to \emph{internal angles} or \emph{relative internal angles}: given a Jordan domain $W\subset \C$ with $0\in W$ consider the unique conformal map $\phi:(\D,0) \to (W,0)$ with $\phi'(0)>0$.
Carathéodory's theorem states that $\phi$ extends continuously to a bijection from $\ov \D$ to $\ov W$.
The internal angle of $w\in \ov W-\{0\}$ is defined as $\arg (\phi^{-1}(w))$ and for $w'\in \ov W-\{0\}$, the relative internal angle from $w$ to $w'$ as $\arg (\phi^{-1}(w')/\phi^{-1}(w))$.

\begin{lemma}\label{lem:rot}
Let $P$ be a polynomial of degree $\geq 2$ and $\Gamma\subset \C$ a Jordan curve.
Assume $P(\Gamma)=\Gamma$ and $P$ is conjugate on $\Gamma$ to a rotation.
Then there is a period one Siegel disk such that $\Gamma$ is contained in it or is equal to its boundary.
\end{lemma}
\begin{proof}
Let $D$ be the bounded complementary component of $\Gamma$.
By the maximum principle, $f$ sends $D$ to itself.
By the argument principle, every point in $D$ has exactly one preimage in $D$.
Conjugating $f$ by a conformal mapping from $D$ to $\D$, we get a bijective holomorphic self-map of $\D$, hence a homography, which is conjugate to a rotation on $\partial \D$.
Such a map is conjugate to a rotation on $\D$ by another homography.
It follows that the invariant open set $D$ is contained in a Siegel disk of period one and contains its center.
The result follows.
\end{proof}

We recall that $a=(a_2,\ldots,a_{d-1})\in\C^{d-2}$.
By $a\to\infty$ we mean that at least one coefficient of $a$ tends to infinity.

\begin{lemma}[Limit of $\rad$ at $\infty$]\label{lem:rto0}
Fix $\theta\in\R-\Q$.
\[\lim_{a\to\infty} \rad(P_a[\theta]) = 0
\]
\end{lemma}
\begin{proof}
From $P'_a = dz^{d-1} + \sum_{k=2}^{d-1} k a_k z^{k-1} + \rho = d(z-c_1)\cdots(z-c_{d-1})$, where $c_j$ are the critical points of $P_a$, we get two facts.
First if the $c_j$ stay bounded then the $a_k$ stay bounded. By contraposition, if $a \to \infty$, then at least one $c_j$ tends to infinity.
Second the product $c_1\cdots c_{d-1}$ is fixed. It follows that at least one $c_j$ tends to $0$.
But critical points cannot belong to rotation domains, so by Koebe's one-quarter theorem, $\rad(P_{a})\leq 4\lvert c_j\rvert$.
\end{proof}

\begin{lemma}[Continuity of $P\mapsto \partial \Delta(P)$]\label{lem:Delta:cont}
  Let $\theta$ have bounded type. For $P\in\cal P[\theta]$,
  the conformal mapping $\phi_P: (\D,0)\to (\Delta(P),0)$ sending $0$ to $0$ and satisfying $\arg \phi'(0)=0$ has a continuous extension to a homeomorphism from $\ov \D$ to $\ov\Delta(P)$ and this extension varies continuously with $P$.
  In particular the function that sends $P\in \cal P[\theta]$ to $\partial \Delta(P)$ is continuous with respect to the Hausdorff distance on compact sets.
\end{lemma}
\begin{proof}
  Since $\partial \Delta(P)$ is a Jordan curve, the Carathéodory extension theorem implies that $\phi$ extends to a homeomorphism from $\ov \D$ to $\ov \Delta(P)$.

  In the proof of \Cref{thm:r:cont} in \cite{C:these}, we proved in particular the Carathéodory continuity\footnote{The set of pointed simply connected open strict subset of $\C$ is endowed with a topology called the Carathéodory topology, recalled near the end of this proof. This is different from the Carathéodory-Torhorst extension theorem.} of $P\mapsto (\Delta(P),0)$.
  However, this is not enough to imply Hausdorff continuity of $P\mapsto \partial \Delta(P)$.

  The theorem of Shishikura (\Cref{thm:shi}) not only states that the Siegel disk is a quasidisk, but also that the quasiconformality constant $K=K(\theta)$ can be taken uniform on $\cal P[\theta]$, see \cite{ZhangGaofei}.

  Now, for a fixed $K$, for any sequence $U_n$ of $K$-quasidisks and any sequence of points $u_n\in U_n$, Carathéodory convergence of $(U_n,u_n)$, \ie convergence on every compact subset of $\D$ of the conformal mapping $\phi_n :(\D,0) \to (U_n,u_n)$ satisfying $\arg \phi_n'(0)=0$, implies uniform convergence of $\phi_n$ (this is proved by using that the $\phi_n$ forms a uniformly equicontinuous family for the Euclidean metric on $\C$), so of its extension to $\ov \D$.
\end{proof}

It will be useful to work on a ramified covering of the parameter space, on which the critical points can be marked and depend holomorphically on the parameter.
\begin{definition}\label{def:Xi}
  Let $\Lambda = \Lambda_d$ be the smooth complex submanifold of $\C^{d-1}$ defined by \[ (c_1,\ldots, c_{d-1})\in \Lambda \iff c_1\, \cdots\, c_{d-1} = \frac{\rho}{(-1)^{d-1} d}
  . \]
  The projection to the $d-2$ first coordinates is an isomorphism to $(\C^*)^{d-2}$.
  Let
  \[\Xi : \Lambda \mapsto \cal P\]
  sending $\lambda = (c_1,\ldots,c_{d-1})$ to the unique monic polynomial $P$ that vanishes at $0$ and such that
  \[ P' = d(z-c_1)\cdots(z-c_{d-1})
  . \]
\end{definition}
Of course, the map $\Xi$ depends on $\rho=e^{i\theta}$.
The critical points of $P=\Xi(c_1,\ldots,c_{d-1})$ are $c_1,\ldots,c_{d-1}$, with multiplicities accounted for by repetitions.
In the sequel, the notation $c_k$ denotes either a complex number, or the function mapping $\lambda\in \Lambda$ to its $k$-th coordinate, the context making it clear which meaning is intended.

\medskip

We recall the definition of holomorphic motions (see \cite{MSS}).
Consider a Riemann surface $X$ (a connected 1-complex dimensional manifold) and a family $A_x$, $x\in X$ of subsets of $\hat \C$.
We say that \emph{$A_x$ undergoes a holomorphic motion with $x\in X$} if there exists a basepoint $x_*\in X$ and a map $h:X\times A_{x_*}\to \RS$ such that:
\begin{itemize}
  \item $\forall x\in X$, the map $h_x: z\in A_{x_*} \mapsto h(x,z)$ is a bijection to $A_x$,
  \item $h_{x_*}$ is the identity,
  \item $\forall z\in A_{x_*}$, $x\mapsto h(x,z)$ is holomorphic.
\end{itemize}
The map $h$ is itself called a \emph{holomorphic motion} with basepoint $x_*$. For any $x'\in X$, the maps $\tilde h(x,z) = h_{x}\circ h_{x'}^{-1}(z)$ provide then a holomorphic motion with basepoint $x'$.
The $\lambda$-lemma of \cite{MSS} states that the maps $h_x$ are quasisymmetric and that $h$ uniquely extends to a holomorphic motion of the sets $\ov A_x$.
Its proof crucially uses the injectivity of the maps $h_x$.

Let us add a not so well-known fact, proved in \cite{Zakeri-fit}, Theorem~6.4 page~1042: if a family of subset of $\RS$ undergoes a holomorphic motion and their interiors are empty, then for any basepoint $x_*\in X$ the holomorphic motion $h$ is unique. Though we will not use this fact, we find interesting to mention it.

For an analytic family of polynomials (or rational maps) $f_x$, $x\in X$, we say that $h$ \emph{commutes with the dynamics} if $\forall x\in X$, $f_x(A_x)\subset A_x$ and $h_x$ conjugates $f_{x_*}$ to $f_x$. This is independent of the choice of the basepoint.

Finally, we say that \emph{$A_x$ locally undergoes a holomorphic motion with $x\in X$} if for every $x'\in X$, there exists an open connected neighborhood $V$ of $x'$ such that $A_x$ undergoes a holomorphic motion with $x\in V$.

The following is partially due to Sullivan, see also \cite{Zakeri-fit,BC1}. 
\begin{theorem}[Rotation domains undergoing holomorphic motion]\label{thm:harmo:1}
  Let $\theta$ be an irrational real number.
  Let $X$ be a Riemann surface (not necessarily simply connected).
  Let $(f_x)_{x\in X}$ a family of maps such that $(x,z)\mapsto f_x(z)$ is defined on an open subset of $X\times\C$, takes values in $\C$ and is analytic in $(x,z)$.
  Assume that these maps all have at $0$ a period one point of multiplier $e^{2\pi i\theta}$ and that there are rotation domains $W_x$ of $f_x$ containing $0$ such that $\partial W_x$ is undergoing a holomorphic motion with $x\in X$.
  Then $x \mapsto \log \rad(W_x)$ is harmonic, and the holomorphic motion $h$ commutes with the dynamics.
  Moreover it has a unique extension to a holomorphic motion to the family $\ov W_x$, $x\in X$, such that every $h_{x}$ is a conformal isomorphism between $W_{x_*}$ and $W_x$.
  Moreover, this extension commutes with the dynamics.
\end{theorem}

\begin{proof}
Locally,\footnote{In \cite{Zakeri-fit}, $X=\D$.} condition~(iv) of the main theorem of \cite{Zakeri-fit} is satisfied\footnote{Using their notation, take $t=x$, $U_t=W_x$, $c_t=0$, $\rho_{\theta,t}=f_x$.} so by (vi) in that same theorem, $x\in X\mapsto\log \rad(W_x)$ harmonic (being harmonic is a local condition).
For every parameter $x'\in X$, choose an open subset $ V\subset X$ containing $x$ and such that there is a conformal isomorphism sending $V$ to $\D$ and $x'$ to $0$. By (ii), the holomorphic motion $h$ with parameter set $V$ and basepoint $x'$ extends to $\ov W_x$ to a holomorphic motion $h$ for which every $h_{x} : W_{x'}\to W_{x}$ is a conformal map fixing $0$.
Since the conformal mapping conjugates $f_x$ to $R_\theta$, it follows that the extended holomorphic motion commutes with the dynamics.
By continuity of $h_{x}$ on $\ov W_{x'}$ (following for instance from the $\lambda$-lemma), the original holomorphic motion of the family $\partial W_x$, $x\in X$, commutes with the dynamics too, when restricted to parameters $x\in V$.
Since $X$ is connected, the commutation relation can be extended to far away pairs $(x,x')$ by a finite cover argument and composition, so the original holomorphic motion $(\partial W_x)_{x\in U}$ commutes with the dynamics on the whole parameter set $X$.

Uniqueness in the statement follows from the following argument.
Fix a basepoint $x'\in X$ and imagine that the original holomorphic motion has two extensions
$h$ and $\hat h$ to holomorphic motions of the family $\ov W_x$, $x\in X$, that $\forall x\in X$, $h_x(0)=0$ and $\hat h_x(0)=0$ and that each $h_{x}$ and $\hat h_{x}$ are conformal isomorphisms from $W_{x'}$ to $W_x$.
Consider the map $g_x = \hat h_{x}^{-1} \circ h_{x} : \ov W_{x'} \to \ov W_{x'}$.
It is a conformal self-mapping of the simply connected subset $W_{x'}$ of $\C$ with non-empty boundary and that coincides with the identity on this boundary, so it is equal to the identity.
\end{proof}

\begin{definition}\label{def:subo}
  In the next lemmas, we will consider one parameter \emph{sub-families} (subordinate families) of the family $\Xi: \Lambda \to \cal P[\theta]$ of \Cref{def:Xi}, \ie families of polynomials $x\in X\mapsto \Xi(\Phi(x)) \in \cal P[\theta]$ where $X$ is a Riemann surface and $\Phi:X\to\Lambda$ is a holomorphic function.
\end{definition}
For convenience we will omit the index $x$ and denote $P= \Xi(\Phi(x))$, $c_n = c_n(\Phi(x))$, etc.
In the sequel we assume that
\begin{center}
    $\theta$ is a bounded type irrational.
\end{center}

The following argument has a short proof but will be used several times so we formulate it as a lemma.

\begin{lemma}\label{lem:cbord}
  Let $x\in X\mapsto P= \Xi\circ\Phi(x)$ be a sub-family of $\Xi$, as per \Cref{def:subo}.
  Assume that $X$ is connected and that there is some $n\in \lib 1,d-1 \rib$ such that $c_n\in \partial \Delta(P)$ for all $x\in X$.
  Then $\partial \Delta(P)$ undergoes a holomorphic motion with $x\in X$.
\end{lemma}
\begin{proof}
  The points in the orbit of $c_n$ are all disjoint since $P$ is conjugates to an irrational rotation on $\partial \Delta(P)$, so this orbit undergoes a holomorphic motion with $x\in X$, so $\partial \Delta(P)$ too by the $\lambda$-lemma.
\end{proof}

\Cref{lem:cbord,thm:harmo:1} are used in several places in this article but the proof of the statement below is the only place where we use holomorphic motions that are global (in therms of the parameter space $X$).

\begin{lemma}\label{lem:holomo:crit}
  Let $x\in X \mapsto P= \Xi\circ\Phi(x)$ be a sub-family of $\Xi$, as per \Cref{def:subo}.
  Assume that $\partial \Delta(P)$ \emph{locally} undergoes a holomorphic motion with $x\in X$.
  Then the indices $n \in \lib 1,d-1 \rib$ such that $c_n\in \partial \Delta$ are the same throughout $X$ and $\partial \Delta$ undergoes a \emph{global} holomorphic motion with $x\in X$.
\end{lemma}
\begin{proof}
  Denote $c_n(x)=c_n(\Phi(x))$, $\Delta_x=\Delta(P)$ and $P_x=P$.
  Assume that for some $n$ and $x_0\in U$, $c^* := c_n(x_0)\in \partial \Delta_{x_0}$.
  Let $h$ be a local holomorphic motion, over some parameter set $V \subset X$ containing $x_0$, for the family $\ov \Delta_x$, $x\in V$ (initially defined on $\partial \Delta_x$, extended to $\Delta_x$ by \Cref{thm:harmo:1}).
  Consider the holomorphic function $g(u) = h_{x}(c^*)$.
  Then the functions $g$ and $x\mapsto c_n(x)$ coincide at $x=x_0$.
  Let us prove that they coincide in $V$.
  Assume by way of contradiction that this is not the case.
  Then for every $a\in\Delta_{x_0}$ close enough to $c^*$, the function $g_a: x\in V \mapsto h_{x}(a)$ is close to $g$, so it would follow that there are values of $x\in V$ such that $g_a(x) = c_n(x)$ by Hurwitz's theorem, which is absurd since no critical point can belong to a Siegel disk.
  Since $c_n(x)\in \partial \Delta_x$ for all $x\in X$, 
  $\partial \Delta_x$ undergoes a holomorphic motion over $X$ by \Cref{lem:cbord}. 
\end{proof}

\Cref{lem:rel:sc}  below is the only place where we use that the holomorphic motion given by \Cref{lem:holomo:crit} is global (on the set $W$ below), and is used only in \Cref{lem:new:crit}.

\begin{lemma}[Relative simple connectivity of holomorphic motion locus]\label{lem:rel:sc}
  Let $x\in X\subset\C\mapsto P= \Xi\circ\Phi(x)$ be a sub-family of $\Xi$, as per \Cref{def:subo}.
  Let $X'$ be the subset of $X$ on which $\partial \Delta(P)$ undergoes a local holomorphic motion.
  Then $X'$ is what we call \emph{relatively simply connected in $X$} \ie for any Jordan domain $W\subset X$ such that $\partial W\subset X'$ then $W\subset X'$.\footnote{Equivalently, the inclusion of $X'$ in $X$ induces an injective map on the level of their fundamental groups. We will not use this equivalent formulation.}
\end{lemma}
\begin{proof}
  Denote $P_x=P$.
  By \Cref{lem:holomo:crit}, the family $\partial \Delta(P_x)$ undergoes a holomorphic motion over every connected component of $X'$.
  Let $\gamma$ be a simple closed loop tracing $\partial W$ in the positive orientation.
  Let $X''$ be the connected component of $X'$ that contains $\partial W$.
  Denote $x^*=\gamma(0)$.
  There is at least one critical point on $\partial \Delta(P_{x^*})$ and without loss of generality we assume that it is $c_1$.
  By \Cref{lem:holomo:crit}, $c_1$ remains in $\partial \Delta(P_x)$ for all $x\in X''$.
  By \Cref{thm:harmo:1} the holomorphic motion of $\partial \Delta(P_x)$ extends to $\ov \Delta(P_x)$ as isomorphisms of $\Delta(P_x)$ for all $x\in U''$, conjugating the restrictions of $P_x$ to $\Delta_x$, in particular fixing $0$.
  When the parameter $x$ varies along $\partial W$ parametrized by $\gamma$, the winding number of $c_1$ relative to $0$ is $0$, otherwise $c_1$ would vanish somewhere in $W$, but it would contradict that $0$ is a the neutral fixed point.
  Note that $c_1(x) = h_{x}(c^*)$ where $c^* = c_1(x^*)$.
  For any two distinct points $z,z'\in \ov \Delta(P_{x^*})$, the winding number of the vector $h_{u}(z)-h_{x}(z')$ as $x$ goes along $\partial W$ is also $0$: indeed by continuity it is independent of $(z,z')$ (since the vector never vanishes), and equals $0$ for $(z,z') = (0,c^*)$ as we just proved.
  Applied to $z=P_{x^*}^n(c^*)$ and $z'=P_{x^*}^{n'}(c^*)$ for any $n\neq n'$ in $\N$, this tells us that the function $x\mapsto P_x^n(c_n(x))-P_x^{n'}(c_n(x))$ has winding number $0$ relative to $0$, as $x$ varies along $\partial W$.
  Hence it cannot vanish in $W$: so the orbit of $c_1$ remains disjoint for all $x\in W\subset X$.
  This orbit thus undergoes a holomorphic motion in $W$, that commutes with the dynamics.
  By the $\lambda$-lemma it extends to the closure of the orbit.
  By continuity the extended motion still commutes with the dynamics.
  For every $x\in W$ the image of $\partial \Delta(P_{x^*})$ is a Jordan curve mapped to itself by $P_x$ and on which $P$ is conjugated to $R_\theta$ on the unit circle.
  By \Cref{lem:rot} it thus bounds a rotation domain of $P_x$.
  Since it also has a critical point on its boundary, this rotation domain is the Siegel disk of $P_x$ at $0$.
\end{proof}

The next result is a consequence of the following rigidity result due to Gaofei Zhang (\cite{ZhangGaofei-TopoPre}, see also \cite{ZhangGaofei-TopoPub}).

\begin{theorem}[Rigidity, \cite{ZhangGaofei-TopoPre}, Theorem~B]\label{thm:Zhang:rigidity}
  Let $f$ be a rational map with a period one Siegel disk of bounded type rotation number such that $\partial \Delta(f)$ is a quasicircle\footnote{This is automatic by further work of Zhang, see \cite{ZhangGaofei}.
  For polynomials this is due to Shishikura, see \Cref{thm:shi}} and such that the post-critical set of $f$ is contained in a finite set union $\ov\Delta(f)$.
  Let $g$ with the same properties and such that there exists an equivalence to $f$ in the following sense: $\exists \phi_1$, $\phi_0$ orientation preserving homeomorphisms of $\RS$ such that:
  \begin{enumerate}
    \item $\phi_0 \circ f = g\circ \phi_1$,
    \item $\phi_0$ and $\phi_1$ coincide on $\Delta(f)$ and is holomorphic,\footnote{This implies that $\Delta(g)$ is the image of $\Delta(f)$ by either map.}
    \item $\phi_0$ is isotopic to $\phi_1$ rel.\ the post-critical set of $f$. 
  \end{enumerate}
  Then $f$ and $g$ are conjugate on $\RS$ by a Möbius map.
\end{theorem}

\begin{note*}
  In \cite{ZhangGaofei-TopoPre}, Theorem~B, it is required that $f$ and $g$ \emph{realize} some specific class of ramified topological self coverings of $\RS$ equal to a rotation on the unit disk. But this is easily seen to be automatic as the author proves in the beginning of the proof of Lemma~3.1.
\end{note*}

Note that in $\Lambda$, Möbius conjugacy classes are conjugacy classes by the finite group generated by the permutation group on the coordinates of $\lambda = (c_1,\ldots,c_{d-1})$ plus the rotation $\lambda\mapsto \nu\lambda = (\nu c_1,\ldots, \nu c_{d-1})$ with $\nu = e^{2\pi i/{d-1}}$. In particular these classes are finite.

\Cref{thm:Zhang:rigidity} implies the following statement:

\begin{lemma}[Rigidity]\label{lem:rigid}
  Let $k\in\lib 1,d-1 \rib$. Denote $Z$ the set of 
  $\lambda\in \Lambda$ for which, denoting $P:=\Xi(\lambda)$ and $c_n=c_n(\lambda)$: $c_1,\ldots,c_k$ belong to $\partial \Delta(P)$ while every other critical point is periodic (this last condition is empty if $k=d-1$).
  Then every $\lambda_a\in Z$ has a neighborhood $V\subset Z$ such that: $\forall \lambda \in V$, if the relative internal angles of the critical points $c_1,\ldots, c_k$ are identical for $\lambda$ and $\lambda_a$ then $\lambda =\lambda_a$.
\end{lemma}

\begin{note*}
  We expect $Z$ to be a countable and disjoint union of dimension $k-1$ topological sub-manifolds of $\Lambda$, homeomorphic to tori, that are not real-differentiable sub-manifolds if $k>1$. 
  In the case $d=3$ and $k=2$, Zakeri proved that $Z$ is a Jordan curve.
  In the case $d>3$ and $k=d-1$, $Z$ is a $d-1$-dimensional torus (this follows from Zhang's theorem and continuity of a classical Blaschke-model surgery, see also Section~4 of \cite{Lim-Rigidity}), so $I_{d-1}$ is a ramified covering of a quotient of the torus.
  We will not need these facts here.
\end{note*}

\begin{proof}
Denote $\lambda_b=\lambda$, $P_a=\Xi(\lambda_a)$ and $P_b = \Xi(\lambda_b)$.
By \Cref{thm:Zhang:rigidity} and discreteness of Möbius conjugacy classes in $\Lambda$, it is enough to prove that for $\lambda_b\in Z$ close enough to $\lambda_a$, then $P_b$ is equivalent to $P_a$ (in the sense defined in the theorem).

First note that, for every $\lambda_b\in Z$ close enough to $\lambda_a$, the period of each critical points $c_k$, the multiplicities of the critical points and the critical relations, are unchanged: on $\partial \Delta$ this is because the relative internal angles of the critical points are the same; the other critical points, are assumed periodic by definition of $Z$ and an attracting basin contains only one cycle.

Denote $\on{pcr}(P_a)$ the post-critical set of $P_a$, i.e.\ the union of the orbits of its critical values. Denote
\[L_a = \on{pcr}(P_a)\cup\ov{\Delta}(P_a) = \on{opc}(P_a) \cup \ov{\Delta}(P_a),\]
where $\on{opc}(P_a)$ is the union of the orbits of the periodic critical points of $P_a$.
The set $\on{opc}(P_a)$ is finite, disjoint from $\ov{\Delta}(P_a)$, and $L_a$ is compact.
Define the same objects with $b$ in place of $a$.

We start the construction of a homeomorphism $\phi_0:\RS\to\RS$ such that:
\begin{enumerate}[a)]
\item\label{item:8} $\phi_0$ sends each critical point  $c_k(\lambda_a)$ to $c_k(\lambda_b)$ and $\phi_0$ conjugates $P_a$ to $P_b$ on the critical orbits of $P_a$,
\item\label{item:4} $\phi_0$ is a conformal isomorphism from $\Delta(P_a)$ to $\Delta(P_b)$.
\end{enumerate}

Define first $\phi_0$ on $\Delta(P_a)$ by taking the unique isomorphism to $\Delta(P_b)$ that sends the fixed point $0\in \Delta(P_a)$ to the fixed point $0\in \Delta(P_b)$, and whose continuous extension maps the $c_k(P_a)$ in $\partial \Delta(P_a)$ to $c_k(P_b)$: this is possible since we assume that the relative internal angles are the same.
On $\partial\Delta(P_a)$ take the continuous extension of $\phi_0$. On $\on{opc}(P_a)$, send $P_a^n(c_k(P_a))$ to $P_b^n(c_k(P_b))$: this is possible since the critical relations are the same.
%
%
Note that
\begin{equation}\label{eq:svdgtnhk}
  \forall z\in L_a,\ P_b\circ \phi_0(z) = \phi_0\circ P_a(z).
\end{equation}

For each point in the critical cycles, and also for $\ov\Delta_a$, choose a smooth Jordan domain containing it, independent of $\lambda_b$, and choose these domains disjoint from each other.
Define $\phi_0$ to be the identity on the complement of the union of these domains.
On each domain containing a point of $\on{opc}(P_a)$, it is easy to define explicitly a homeomorphic extension of $\phi_0$, that is moreover close to identity when $\lambda_b$ is close to $\lambda_a$.
On the domain containing $\Delta(P_a)$, the same can be done, see \Cref{lem:close} in \Cref{app}.

We claim that if $\lambda_b$ is close enough to $\lambda_a$ (in particular, $\phi_0$ is close to the identity), then there exists a homeomorphism $\phi_1: \RS\to \RS$ such that
\begin{enumerate}[a),resume]
\item\label{item:7} $\phi_0\circ P_a = P_b \circ \phi_1$.
\item\label{item:11} $\phi_1$ coincides with $\phi_0$ on $L_a$,
\item\label{item:6} $\phi_1$ is isotopic to $\phi_0$ relative to\footnote{Denoting the isotopy $t\mapsto\phi_t$, this means $\forall z\in L_a$, $t\mapsto\phi_t(z)$ is constant.} $L_a$,
\end{enumerate}

The existence of a homeomorphism $\phi_1$ satisfying \Cref{item:7} and moreover close to identity is classical, but for convenience we provide a proof in \Cref{app} as \Cref{lem:phi0phi1}.

To prove \Cref{item:11}, by \Cref{eq:svdgtnhk} we have
$\ov L_a $: $P_b \circ \phi_1 = \phi_0 \circ P_a  = P_b \circ \phi_0$, i.e.\ 
\begin{equation}\label{eq:lknhzcfgs}
 \forall z\in L_a,\ P_b \circ \phi_1(z) = P_b \circ \phi_0(z).
\end{equation}
Given any $z\in L_a$ then for $\lambda_b$ close enough to $\lambda_a$, (hence both $\phi_0$ and $\phi_1$ close to the identity), $\phi_0$ and $\phi_1$ necessarily coincide $z$.
Hence for $\lambda_b$ close enough to $\lambda_a$ $\phi_0$ and $\phi_1$ coincide on the finite set $\on{opc}(P_a) \cup \{0\}$.
By a connectedness argument, eq.~\eqref{eq:lknhzcfgs} implies that $\phi_0$ and $\phi_1$, which coincide at $0$, also coincide on $\Delta(P_a)$.

Now, $\phi_0$ and $\phi_1$ are close to the identity and equal on $L_a$, hence \Cref{item:6} holds (by a result of Dyer and Hamstrom: \cite{DyHa}, Theorem~1, applied to $\phi_1^{-1}\circ\phi_0$ on the complement of $\Delta_a$, 
there is an isotopy from $\phi_0$ to $\phi_1$ rel.~$\ov\Delta(P_a)$ that remains close to identity; then it is easy to modify this isotopy near points of $\on{opc}(P_a)$).
\end{proof}

\begin{lemma}\label{lem:new:crit}
Let $x\in X\mapsto P= \Xi\circ\Phi(x)$ be a sub-family of $\Xi$, as per \Cref{def:subo}.
Assume that there exists 
\begin{itemize}
  \item $k\in\N$ with $2\leq k\leq d-1$,
  \item $m_2$, \ldots, $m_{k-1} \in \N$ (if $k=2$ this list has no entry)
\end{itemize}
such that $\forall x\in X$,
\begin{itemize}
\item $c_{1+k},\ldots,c_{d-1}$ are periodic,
\item $\forall n\in\N$ with $2\leq n \leq k-1$, $P^{m_n}(c_n)=c_1$.
\end{itemize}
Consider the subset $\tilde X$ of $X$ of parameters for which $c_k \notin \partial \Delta(P)$.
Let $x_*\in X$ and assume $x_*$ belongs to the boundary of a connected component $\tilde X_0$ of $\tilde X$.
Then for every neighborhood $V$ of $x_*$ there exists $x'\in V \cap \partial \tilde X_0$ such that $c_{k}$ (which belongs to $\partial \Delta(P)$) is eventually mapped to $c_1$.
\end{lemma}
\begin{proof}
Note that $\tilde X$ and $\tilde X_0$ are open.

We claim that $\forall x\in \tilde X$, $c_1\in\partial \Delta(P)$.
Indeed, $\partial \Delta(P)$ contains at least one critical point $c_n$.
Since $\partial \Delta(P)$ contains no periodic critical point, we have $n\leq k$. By hypothesis on $\tilde X$, $n\neq k$, so actually $n<k$.
If $n=1$ we are done.
Otherwise, we can use $c_1 = P^{m_n}(c_n)$ and $P(\partial \Delta(P)) \subset \Delta(P)$.

By \Cref{lem:cbord} applied to $n=1$, $\partial \Delta(P)$ undergoes a holomorphic motion over every connected component of $\tilde X$.
By \Cref{lem:holomo:crit}, the set $I$ of indices of the critical points on $\partial \Delta(P)$ remains the same throughout $\tilde X_0$.
As we saw, $1\in I\subset\{1,\ldots, k-1\}$.
For every  $n\in I$, the relative angle from $c_n$ to $c_1$ on $\partial \Delta(P)$ remains equal to $m_n 2\pi\theta$ throughout $\tilde X_0$ because of the relation $P^{m_n}(c_n) = c_1$.
This is also the case on $\partial \tilde X_0$, by continuity.
For every $x\in \partial \tilde X_0$, the set indices of the critical points in $\partial \Delta(P)$ is $I\cup\{k\}$.

We may work in a chart of $X$, with image $U\subset \C$.
Denote $u_*$ the image of $x_*$.
Choose $\eps>0$ such that $\ov B:=\ov B(u_*,\eps)$ is contained in the image of $V$.
Denote $\tilde U$ and $\tilde U_0$ the images of $\tilde X$ and $\tilde X_0$ in the chart ($\tilde U_0$ is not necessarily connected but that does not impact the proof).
From now on we use the sub-family of $u\mapsto P_u$ indexed by $u\in U$.
By \Cref{lem:rigid} and the previous paragraph, on $B \cap (\partial \tilde U_0)-\{u_*\}$, the relative angle from $c_k$ to $c_1$ cannot equal the value it takes at $u_*$.

Let $L$ be the connected component containing $u_*$ of $\ov B \cap \partial \tilde U_0$.
We claim that $L$ is not reduced to a point.
Indeed, if $L$ were reduced to $u_*$ then there would exist a Jordan curve $\gamma$ in $B\cap \tilde U_0$ separating $u_*$ from $\partial B$ (see \Cref{sec:sep}: apply \Cref{cor:sep} to $V=U_0$ and $x = u_*$).
Denote $U'$ the subset of $U$ where there is a local holomorphic motion of $\partial \Delta(P)$.
In particular $\tilde U_0\subset\tilde U\subset U'$.
The curve $\gamma$ would be contained in $B\cap U'$ and \Cref{lem:rel:sc} would imply $u_*\in U'$, contradicting that in $U'$ the set of indices of critical points on $\partial \Delta$ cannot change locally (\Cref{lem:holomo:crit}).

Recall that for $n\leq k-1$, the relative internal angle from $c_n$ to $c_1$ is fixed.
By continuity the image of $L$ by the relative angle from $c_k$ to $c_1$, seen as a function of $u\in L$, is a connected subset of $\R/2\pi\Z$, hence an interval (or the whole set $\R/2\pi\Z$).
This interval is not reduced to a point, for otherwise the relative angle would be constant on $L$.
It thus contains a value of the form $m\theta$ for some $m\in \N$ (the set $\theta\N$ is dense in $\R/2\pi \Z$).
For a parameter yielding this value, $P^k(c_k)=c_1$.
\end{proof}

Finally we will need a result about activity of critical points. Given a sub-family $P_x=\Xi\circ\Phi(x)$ with $\Phi:X\subset \C \to \Lambda$, we recall that a critical point $c_n$ is called \emph{passive} on $X$ (relative to $\Phi$) when the family of function $x\mapsto P_x^n(c_n(\Phi(x)))$ is normal on $X$.
Otherwise it is called \emph{active} on $X$.
The \emph{activity locus} is the set of points $x\in X$ that do not have a neighborhood on which the restriction of $\Phi$ is passive. 

Dujardin and Favre proved the following, as a particular case of Theorem~4 in \cite{DF}.\footnote{Where it is stated for $X\subset \C$; the version for a general Riemann surface $X$ is an almost direct consequence; however we will only use it locally; they call linearization domain the set of points eventually falling in the Siegel disk; there is a slight mistake in the statement in \cite{DF} about the nature of the set $P$ in case~\eqref{case:c}, a correction is brought \cite{Chio-Roeder}, Section~2.5 page~508; for the weaker version given here, this makes no difference.}
\begin{theorem}[passive critical points]\label{thm:passive}
If $c_n$ is passive on the whole set $X$ (relative to $\Phi$) then denote $P$ the set of parameters in $X$ for which $c_n$ is preperiodic.
\begin{enumerate}
    \item\label{case:a} If $P=\emptyset$ then there is a holomorphic motion of the closure of the orbit of $c_n$ that commutes with the dynamics.
    \item\label{case:b} If $P=X$ then there are $m\neq n\in \N^2$ such that the relation $P_x^m(c)=P_x^n(c)$ holds for every $x\in X$.
    \item\label{case:c} If $P\neq \emptyset$ and $P\neq X$ then either there exists a persistently attracting cycle attracting $c_n$, or the orbit of $c_n$ eventually falls in the interior of a persistent\footnote{I.e.\ the Siegel disk of a persistent indifferent cycle of irrational rotation number.} periodic Siegel disk.
\end{enumerate}
\end{theorem}
The first case is a direct application of the $\lambda$-lemma of \cite{MSS} and of a continuity argument.

\begin{lemma}\label{lem:active}
Let $x\in X\mapsto P= \Xi\circ\Phi(x)$ be a sub-family of $\Xi$, as per \Cref{def:subo}.
Let $X'$ be the subset of $X$ where $\partial \Delta$ has a local holomorphic motion. Let $x_0\in X-X'$.
For every $n$ such that $c_n(\Phi(x_0))\in \partial \Delta_{x_0}$, then $x_0$ is in the activity locus of $c_n$ (relative to $\Phi$).
\end{lemma}
\begin{proof}
Assume the contrary and let $V\subset X$
be such that $c_n$ is passive on $V$.
Apply \Cref{thm:passive} on $X=V$.
We cannot be in case~\eqref{case:b} because $c_n$ is not preperiodic for $x=x_0$.
Nor can we be in case~\eqref{case:c} because for $x=x_0$, $c_n$ is neither in an attracting basin nor eventually falls in a Siegel disk.
We are thus in case~\eqref{case:a} and there is a holomorphic motion of the closure of the orbit of $c_n$ that commutes with the dynamics.
For $x=x_0$, this closure is $\partial \Delta_{x_0}$.
For every $x\in V$, the closure is thus a Jordan curve on which $P_x$ is conjugated to an irrational rotation.
By \Cref{lem:rot} it is contained in a Siegel disk or equal to the boundary of a Siegel disk.
Since it contains a critical point, it is the second possibility that holds.
\end{proof}

\subsection{Proof}\label{sub:pf:main}

We assume $d\geq 3$ for otherwise there is nothing to prove ($\cal P_2[\theta]= I_{2,1}[\theta]$ is a point).
Below we denote $\cal P=\cal P[\theta]$, $\Lambda = \Lambda_d$, $ I_k =  I_{k}[\theta]$. 

By \Cref{thm:r:cont} the function $P\in \cal P\mapsto \rad (P)$ is continuous.
By \Cref{lem:rto0} it tends to $0$ when $P \to \infty$.
Hence it reaches its maximum at some (not necessarily unique)
\[P_* \in \cal P.\]
Let $\eta>0$.
We will prove that 
\begin{equation}\label{eq:0}
  \exists P\in  I_{d-1}\text{ such that }\rad(P)\geq\rad(P_*)-\eta
  .
\end{equation}
Since $ I_{d-1}$ is compact,\footnote{See the paragraph before the main theorem in the introduction.} the main theorem will follow by taking limits as $\eta\to 0$.

We recall that by Shishikura's theorem (\Cref{thm:shi}), for every $P\in \cal P$, there is at least one critical point on $\partial \Delta(P)$ (this also follows from \cite{GS}).

\bigskip

We first cover the case $d=3$, because it is simpler: we will directly prove that $\exists P\in  I_{d-1}\text{ such that }\rad(P)\geq\rad(P_*)$ (hence $\rad(P)=\rad(P_*)$).
Note that $\cal P$ and $\Lambda$ have complex dimension one.
If $P_*\in  I_{2}$ then we are done.
Otherwise $P_*\in  I_1$, which is an open subset of $\cal P$.
Denote $P_\lambda=\Xi(\lambda)$,  $\cal J_1=\Xi^{-1}( I_1)$ (it is an open subset of $\Lambda$), $\cal J_2=\Xi^{-1}( I_2)$ and choose $\lambda_*\in \Lambda$ such that $P_* = P_{\lambda_*}$.
By definition of $ I_1$, for every parameter in $\cal J_1$, there is a unique $n\in\{1,2\}$ such that $c_n \in \partial \Delta(P)$ (in particular, $c_1 \neq c_2$).
By continuity of $P\mapsto \partial \Delta(P)$, $n$ is constant on any connected component of $\cal J_1$.
By \Cref{lem:cbord} applied to such a component and \Cref{thm:harmo:1} applied to $W=\Delta(P)$, the function $\lambda\mapsto \log \rad(P_\lambda)$ is harmonic on every connected component of $\cal J_1$, hence on $\cal J_1$.
(Here we only use \Cref{lem:cbord,thm:harmo:1} locally.)
It is either locally constant near $\lambda_*$ or not, but that will not change the argument.
Recall also that $\log \rad$ tends to $0$ at infinity in $\cal P$.
By the maximum principle for harmonic functions, there exists $\lambda_1$ in the boundary of the connected component $A$ of $\cal J_1$ containing $\lambda_*$, for which $\log \rad(P_{\lambda_1})\geq \log \rad (P_*)$.
Since $\cal J_1$ is open in $\Lambda$ and $\cal J_1, \cal J_2$ is a partition of $\Lambda$, it follows that $\partial A\subset \cal J_2$, hence $P_{\lambda_1}\in  I_2$. Q.E.D.

\bigskip

We now treat the case $d\geq 4$ (the argument will also work for $d=3$). We proceed in two steps.

\loctitle{Step 1.} 

\begin{proposition}\label{prop:step:1}
  For every $\eta>0$ and $0\leq q \leq d-2$, there exists $P\in\cal P$ such that $\rad(P)\geq\rad(P_*)-\eta$ and $P$ has at least $q$ periodic critical points (counted with multiplicity).
\end{proposition}
\begin{proof}
  We proceed by induction.
  The case $q=0$ is trivial (take $P=P_*$).
  Assume that $q>0$ and that the case $q-1$ is proved.
  Let $\eps>0$.
  By the induction hypothesis, $\exists P_0\in\cal P$ such that $P_0$ has at least $q-1$ periodic critical points and $\rad(P_0)\geq\rad(P_*)-\eps$.
  We can reindex the critical points and choose $\lambda_0=(c_1,\ldots,c_{d-1})\in\Lambda$ such that $P_0=\Xi(\lambda_0)$ and for every $1\leq n\leq q-1$, $c_n$ is periodic of period denoted $p_n$.
  
  Let $A$ be the algebraic subset of $\Lambda$ (which is itself an algebraic subset of $\C^{d-1}$), defined by the conditions $P^{p_n}(c_n)=c_n$ for the values of $n$ above.
  There are at most $d-3$ conditions so its irreducible components have dimension at least $1$.
    
  Consider any irreducible $1$-dimensional algebraic subset $A_1$ of $A$ containing $\lambda_0$:
  \[ A_1 \subset A \subset \Lambda \subset \C^{d-1}
  .\]
  Such a set exists, and cannot be bounded.
  It can be desingularized via a surjective
  holomorphic map 
  \[ \Phi : X \to A_1
  \]
  from a compact connected Riemann surface $X$ minus finitely many points (at least one) to $\C^{d-1}$ taking values in $A_1$ and tending to infinity (in $\C^{d-1}$) at the punctures of $X$, and such that for any non-singular point of $A_1$, the preimage has exactly one element.
  We denote $x_0\in X$ such that $\Phi(x_0)=\lambda_0$.

  In the family parametrized by $x\in X$, consider the set $X'$ where there is local holomorphic motion of $\partial \Delta(P)$.
  If $x_0\in X'$ then, since by \Cref{thm:harmo:1} the function $x \mapsto \log \rad (\Xi\circ\Phi(x))$ is harmonic on $X'$ (here we only use \Cref{thm:harmo:1} locally) and tends to $-\infty$ at the punctures of $X$, there exists $x_{0'}$ in the boundary in $X$ of $X'$ such that, denoting $P_{0'}=\Xi\circ\Phi(x_{0'})$, $\log \rad(P_{0'}) \geq \log \rad (P_0) \geq \log \rad (P_*)-\eps$. In that case we replace $x_0$ by $x_{0'}$.

  We can thus assume without loss of generality that $x_0\notin X'$.
  By \Cref{lem:active} every critical point on $\partial \Delta(P_0)$ is active at $x_0$ for the family $\Xi\circ\Phi$.
  Since there is at least one critical point (actually, at least two) on this boundary, we choose one and, up to reindexing, it is $c_q$.
  By a classical normal family argument dating from Fatou and Julia,\footnote{See for instance \cite{Be}, Lemma~3.1.8, where the word \emph{or} means the reader has the choice between the two properties.} there are parameters $x_1$, arbitrarily close to $x_0$, such that $c_q$ is periodic.
  By choosing $x_1$ close enough to $x_0$, given $\eps'>0$ we can ensure by continuity of $P\mapsto \rad(P)$ that $\rad (P_1) \geq \rad (P_0) - \eps'$, where $P_1=\Xi\circ\Phi(x_1)$.
  
  By choosing $\eps+\eps' \leq \eta$, this proves heredity of the induction hypothesis.
\end{proof}

\loctitle{Step 2.}

\begin{proposition}\label{prop:step:2}
  For every $\eta>0$ and $1\leq k\leq d-1$, there exists $\lambda\in \Lambda$ and $m_2$, \ldots, $m_k$ (if $k=1$, this list is empty) such that, denoting $P = \Xi(\lambda)$:
  \begin{enumerate}
    \item $c_1\in\partial \Delta(P)$, 
    \item $\forall 2\leq n \leq k$, $c_n\in \partial \Delta(P)$ and $f^{m_n}(c_n) = c_1$,
    \item\label{item:star} every other critical point is periodic,
    \item $\rad(P)\geq \rad(P_*)-\eta$.
  \end{enumerate}
  Point \eqref{item:star} concerns $c_n$ for $n>k$; if $k=d-1$ this point is empty.
\end{proposition}
\begin{proof}
This is proved by induction on $k$. The case $k=1$ follows from \Cref{prop:step:1} with $q=d-2$ and from the fact that there is always a critical point on $\partial \Delta(P)$.
We may assume it is $c_1$ up to reindexing.

Let $2\leq k\leq d-1$. Choose $\eps>0$.
We assume the proposition is proved for $k-1$, and we let $\lambda_0\in\Lambda$ such that the claim holds for $k-1$ with
\[\rad(P_0)\geq \rad(P_*)-\eps \]
where $P_0= \Xi(\lambda_0)$.
For $k \leq n\leq d-1$, denote $p_n$ the period of $c_n$ for $\lambda_{0}$.

We start by releasing the condition on $c_k$.  
The $d-3$ conditions $f^{m_n}(c_n) = c_1$ for $2\leq n\leq k-1$ and $f^{p_n}(c_n)=c_n$ for $n\geq k+1$ define an algebraic subset
\[A \subset \Lambda\subset \C^{d-1}\]
of dimension at least $1$.

Exactly as in the proof of \Cref{prop:step:1}, we choose any irreducible $1$-dimensional algebraic subset $A_1$ of $A$ containing $\lambda_0$ (this exists and is non-bounded in $\C^{d-1}$) and a surjective desingularization:
\[ \Phi : X \to A_1
\]
where $X$ is a compact Riemann surface minus finitely many points where $\Phi$ tends to infinity in $\C^{d-1}$.
We denote $x_0\in X$ such that $\Phi(x_0)=\lambda_0$.

For every parameter $\lambda\in A$, $c_1\in\partial\Delta(P)$ or $c_k\in\partial\Delta(P)$. 
Indeed, there is at least one critical point $c_n$ in $\partial \Delta(P)$, and either $n=1$ and we are done, or $n>1$ in which case: it cannot be $>k$ for, by definition of $A$, it would be periodic but there are no periodic point on $\partial \Delta(P)$ as we already saw many times; so either $n=k$ or $n<k$ in which case $c_n$ it is eventually mapped to $c_1$, so $c_1\in \partial \Delta(P)$.

Since for $\lambda_0$, $c_k$ is periodic, hence superattracting, it follows that for all parameters $\lambda\in A$ close to $\lambda_0$, $c_k\notin \partial\Delta(P)$, so $c_1\in \partial\Delta(P)$ for those parameters. 

Consider the sub-family parametrized by $X$.
Let $x_0\in X$ be such that $\Phi(x_0)=\lambda_0$.
Consider the set $\tilde X\subset X$ of parameters for which $c_k\notin\partial \Delta(P)$.
It is open by continuity of $P\mapsto \partial \Delta(P)$.
By the previous paragraph, $x_0\in \tilde X$. Let $\tilde X_0$ be the connected component of $\tilde X$ that contains $x_0$.

On $\tilde X$, we have $c_1\in \partial \Delta(P)$ hence by \Cref{lem:cbord} there is holomorphic motion of $\partial \Delta(P)$ on every connected component of $\tilde X$ (having local holomorphic motion is enough for this proof). So by \Cref{thm:harmo:1} the function $x \mapsto \log \rad (\Xi\circ\Phi(x))$ is harmonic on $\tilde X$.
(Again, we only use \Cref{lem:cbord,thm:harmo:1} locally.)
Since this function tends to $-\infty$ at the punctures of $X$, it follows that there exists $x_{1}$ in the boundary in $X$ of $\tilde X_0$ such that, denoting $P_{1}=\Xi\circ\Phi(x_{1})$, $\log \rad(P_{1}) \geq \log \rad (P_0)$.
In $\tilde X_0$, by the holomorphic motion and by \Cref{lem:holomo:crit}, the critical points $c_1$, \ldots, $c_{k-1}$ (recall $k\geq 2$) remain in $\partial \Delta(P)$, so by continuity this is the case too for any parameter on $\partial \tilde X_0$, like $x_{1}$.

Since $\tilde X$ is open, for parameter $x_1$ we have $c_k \in \partial \Delta(P)$.
We then apply \Cref{lem:new:crit}: there exist parameters $x_2\in\partial \tilde X_0$ arbitrarily close to $x_1$ such that $c_k\in\partial \Delta(P)$ and $c_k$ is eventually mapped to $c_1$.
Since, as we just saw, the critical relations are preserved on $\partial \tilde X_0$, we get the desired relations.  

Given $\eps'>0$, for $x_2$ close enough to $x_1$, we can ensure by continuity of $P\mapsto \rad(P)$ that $\log \rad(P_2)\geq \log \rad(P_1)-\eps'$, denoting $P_2 = \Xi\circ\Phi(x_{2})$.
Choosing $\eps'+\eps\leq \eta$, this proves heredity of the induction hypothesis.
\end{proof}

\Cref{prop:step:2} applied to $k=d-1$ gives us \cref{eq:0}, which, as we saw, proves the main theorem.


\section{Auxiliary statements}\label{app}

We prove here a few results that we consider as known or well-known, yet for which we have difficulties to find direct references.

\subsection{Plane topology}\label{sec:sep}

The objective is to justify \Cref{cor:sep} below.
We start by a few generalities.

\begin{lemma}\label{lem:top1}
  Let $K$ be a compact space, $C$ a connected component of $K$ and $U$ an open subset of $K$ containing $C$.
  Then there exists a partition of $K$ into two closed (hence open) subsets $K_1$ and $K_2$, with $C\subset K_1 \subset U$.
\end{lemma}
\begin{proof}
  This is for instance a corollary of \cite{Mo}, Theorem~3 page~85.
\end{proof}

\begin{lemma}\label{lem:sep}
  Let $K$ be a compact subset of the plane and $x\in K$.
  Assume that $\{x\}$ is a connected component of $K$. Then for all $r>0$ there exists a simple closed curve contained in $B=B(x,r)$, disjoint from $X$ and separating $z$ from $\partial B$.
\end{lemma}
\begin{proof}
By \Cref{lem:top1} there exists a partition of $K$ into two closed subsets $K_1$ and $K_2$ such that $x\in K_1 \subset B(x,r/2)$.
The union $F:=\partial B\cup K_1$ is disjoint. The  connected component of $F$ containing $x$ is thus contained in $K_1$, hence in $K$, so equals $\{x\}$.

In \cite{Newman}, Theorem~3.3 page~143, is proved that two different connected components $F_1$ and $F_2$ of a closed subset $F$ of the Riemann sphere can be separated by a simple closed curve. Apply this to our $F$, to the component $F_1$ of $F$ containing $\partial B$ and to $F_2=\{u_*\}$.
\end{proof}

A counter-example if $K$ is not compact is given by $K=O \cup (\setof{1/n}{n\geq 1}\times[0,1]) \subset \R^2$ where $O=(0,0)$ is the origin.

\medskip

\begin{corollary}\label{cor:sep}
  Let $V$ a connected open subset of the plane and $x\in \partial V$.
  Assume that there exists $\eps>0$ such that the connected component containing $x$ of $\ov B(x,\eps)\cap \partial V$ is a single point. 
  Then there exists a simple closed curve $\gamma$ contained in $V$ and separating $x$ from $\partial B(x,\eps)$.
\end{corollary}
\begin{proof}
Since $x\in\partial V$, there exists $x'\in V\cap B(x,\eps)$.
Let $\gamma$ be given by \Cref{lem:sep} applied to $r=|x'|$.
Let us check that $\gamma\subset V$.
Since $x\in \partial V$ and $V$ is open and connected, we can join by a curve $\delta\subset V$ the point $x'$ to a point in the bounded component of the complement of $\gamma$.
So $\gamma$ and $\delta$ intersect, thus $\gamma$ has at least one point in $V$ and since it does not intersect $\partial V$, it is contained in $V$.
\end{proof}

\subsection{Homeomorphisms close to the identity}\label{app:close}

\begin{lemma}\label{lem:close}
  Let $A\subset\C$ be an annulus bounded by the images of two Jordan curves $\gamma_1,\gamma_2:\R/\Z\to\C$.
  For every $\eps>0$ there exists $\eta>0$ such that for all Jordan curves $\delta_1,\delta_2:\R/\Z\to\C$ satisfying $\forall i\in\{1,2\}$, $\|\delta_i-\gamma_i\|_\infty<\eta$, then there is a homeomorphism $\phi$ from $A$ to the annulus $B$ bounded by $\delta_1$ and $\delta_2$ such that
  \begin{itemize}
      \item $\phi$ extends $\delta_1\circ\gamma_{1}^{-1}$ and $\delta_2\circ\gamma_{2}^{-1}$,
      \item $\|\phi-\on{id}\|_\infty<\eps$.
  \end{itemize}
\end{lemma}
\begin{proof}
  This is stated in Section~2 of \cite{DyHa} as a consequence of the existence of a canonical way to extend $\delta_1\circ \gamma_1^{-1}$, $\delta_2\circ \gamma_1^{-1}$ into a homeomorphism from $A$ to $B$.
 
  For a modern reader not specialized in this field, it is not completely obvious to infer a formal proof from their indications,\footnote{Also, the notion of $0$-regular convergence of compact subsets $t_i$ of the plane is supposed to be defined in their Section~3 but the definition is actually omitted (they define $0$-regular mappings in Section~4 but this is a different, though related, notion). We found it in \cite{White}, which cites \cite{Whyburn} as the origin of the notion. In our context we take $t_i$ to be the whole boundaries of the annuli and $0$-regular convergence is equivalent to uniform local connectivity.} so we give here a few complementary arguments and weaken slightly their statements for simplicity.

  As in \cite{DyHa}, we start with a version for disks instead of annuli:
  there exists a continuous map $T$ that takes as argument a Jordan curve $\gamma: \partial\D \to \C$ such that $0$ is in the bounded complementary component, and outputs an extension $T(\gamma):\ov\D\to\C$ of $\gamma$ such that $T(\gamma)$ is a homeomorphism to the closure of the Jordan domain $D_\gamma$ bounded by $\gamma$.
  The idea to define $T$ is to pre-compose the conformal mapping $\phi_\gamma:\D\to D_\gamma$ sending $0$ to $0$ for which $\phi_\gamma'(0)>0$, with a radial map $(r,\theta)\mapsto(r,f(\theta))$ in polar coordinates, where $f$ is chosen so that $T(\gamma)$ restricts to $\gamma$.
  The fact that $T$ is continuous follows from a uniform version of the theorem of Carathéodory, Osgood and Taylor on the boundary continuity of the conformal maps of Jordan domains: see \cite{Po}, Theorem~2.11 page~26. 

  We then slightly depart from their line of arguments with the following statement:
  there exists a continuous map $T_a$ that takes as argument a Jordan curve $\gamma: \partial\D \to \C$ such that $B(0,1/2)$ is in the bounded complementary component, and outputs an extensions $T_a(\gamma):\ov R\to\C$ of $\gamma$, with $R=\setof{z\in\C}{1/2<|z|<1}$, such that $T_a(\gamma)$ is a homeomorphism to the annulus $A_\gamma$ bounded by $\gamma$ and $\partial B(0,1/2)$ and is the identity on $\partial B(0,1/2)$.
  For this, we use that the conformal map $\phi_\gamma:\D\to D_\gamma$ used in the previous paragraph is differentiable at the origin.
  Consider the blow up map $\psi:\C^*\to\C-B(0,1/2)$ defined in polar coordinates by $(r,\theta)\mapsto ((r+1)/2,\theta)$.
  Consider the Jordan curve $\tilde\gamma = \psi^{-1}\circ \gamma$.
  The map $\psi\circ\phi_{\tilde\gamma}\circ\psi^{-1}$ coincides with identity on $\partial B(0,1/2)$.
  One then compensates this map by pre-composing it with a self-map of $R$ which in polar coordinates is of the form $(r,\theta)\mapsto (r,f_\theta(r))$ where for every $\theta$, $f_\theta$ is affine and $f_\theta(1/2)=\theta$.
  
  To apply this tool to the statement we are addressing, by first applying a homeomorphism of $\C$, 
  we may assume that $\gamma_1$ and $\gamma_2$ are angular parametrizations of the circles of center $0$ and respective radii $1/2$ and $2$.
  Then we can apply separately the above statement to the annuli between $\gamma_i$ and $\partial B(0,1/2)$ (take $\eta<1/2$ and join the two maps thereby constructed.
  Continuity of the construction then gives the result.
\end{proof}

\begin{lemma}\label{lem:phi0phi1}
Let $P_a$ be a polynomial of degree $\geq 2$.
For every $\eta>0$, there exists $\eta'>0$ such that for all polynomial $P_b$, for all homeomorphism $\phi_0$ of $\RS$, if
\begin{itemize}
  \item $\|P_b-P_a\| <\eta'$,
  \item $\|\phi_0-\on{id}\| <\eta'$,
  \item $\phi_0$ sends the set of critical values of $P_a$ to the set of critical values of $P_b$
\end{itemize}
then there exists a homeomorphism $\phi_1$ of $\RS$ such that
$\phi_0 \circ P_a = P_b \circ \phi_1$ and $\|\phi_1-\on{id}\|<\eta$.
\end{lemma}
\begin{proof}
Separate the critical values $v$ of $P_a$ by Jordan domains $U_v$ independent of $P_b$.
The connected components of the preimage of the union of the $U_v$ are Jordan domains $W_c$ containing the critical and co-critical points $c$ of $P_a$, each containing a unique one. 
The restriction $P_a:\ov W_c\to \ov U_v$, with $v=P_a(c)$, is equivalent to $z\in\ov\D\mapsto z^{p_c}\in\ov\D$ for some $p_c\geq 1$.
In other words, there exists holomorphic maps $t_{a,c}:\ov W_c\to\ov \D$ and $s_{v}:\ov U_v\to\ov \D$ such that $(z\mapsto z^{p_c})\circ t_{a,c} = s_{v}\circ P_a$.
The map $\phi_0^{-1}\circ P_b$ is close to $P_a$ and shares the same critical values.
For $P_b$ close to $P_a$, there is still a unique critical value $v'\in U_v$.
There is a deformation $W'_c$ of $W_c$ such that $\phi_0^{-1}\circ P_b(W'_c) = U_v$ (note that we took the same range $U_v$ as for $P_a$, and that we index it with $v$, not $v'$).
Since $U_v$ is simply connected, this implies that the restriction $P_b: \ov W'_c\to \ov U_v$ is still equivalent to $z\in\ov \D\mapsto z^{p_c}\in\ov \D$, via a map $t_{b,c}:\ov W'_c\to\ov \D$, close to $t_{b,c}$ and such that $(z\mapsto z^{p_c})\circ t_{b,c} = s_{v}\circ P_b$ (we can take the same $s_v$ for $P_a$ and $P_b$ because they have the same critical value $v$).
Note that if $c\in W_c$ is a critical point of $P_a$, then the critical point $c'$ of $P_b$ in $W'_c$ is close to $c$, and that the set $W'_c$ is indexed by $c$, not $c'$.
Let $C = \RS - \bigcup_v U_v$. 
As $z$ varies in the compact set $C$, the fiber $P_a^{-1}(z)$ consists in $d$ point that vary continuously with $z$. Their mutual spherical distance remains bigger than some constant $d_a>0$.
For $P_b$ is close enough to $P_a$, $\forall z\in C$, the fiber $(\phi_0^{-1}\circ P_b)^{-1}(z)$ is at distance $<d_a/10$ from the fiber $P_a^{-1}(z)$.
For $w\in P_a^{-1}(z)$, we define $\phi_1(w)$ as the unique point in $(\phi_0^{-1}\circ P_b)^{-1}(z)$ at distance $<d_a/10$ of $w$.
For $\ov W_c$, we choose $\phi_1 = t_{b,c}^{-1} \circ t_{a,c}$, whose distance to identity is arbitrarily small provided $P_b$ is close to $a$.
If $P_b$ is close enough to $P_a$ then this distance is $<d_a/10$ from the identity, in particular it coincides with the values of $\phi_1$ already defined on $\partial W_c$.
We get a continuous extension of $\phi_1$ on $\RS$.
Moreover for all $\eps>0$ (unrelated to $d_a$), we can ensure that the distance from $\|\phi_1-\on{id}\|_{\infty}<\eps$ for the spherical distance, provided $P_b$ is close enough to $P_a$.
From the definition, it follows that the desired relation $\phi_0 \circ P_a = P_b \circ \phi_1$ holds.
Let us prove that $\phi_1$ is injective: if $\phi_1(z) = \phi_1(z')$ then by $\phi_0 \circ P_a = P_b \circ \phi_1$ and injectivity of $\phi_0$, $z$ and $z'$ are in the same fiber of $P_a$; since $\phi_1$ is close to identity, $z$ and $z'$ are close to each other; these two facts imply that $z$ and $z'$ are close to some critical point of $P_a$; then the definition of $\phi_1$ near such critical point implies $z=z'$.
Any continuous injection of the sphere to itself is surjective, and this applies to $\phi_1$.\footnote{Surjectivity can also be justified directly: near critical and co-critical points, it follows from the explicit formula for $\phi_1$, away from them, because the fibers $P_a^{-1}(z)$ and $P_b^{-1}(z)$ have the same cardinality, are at distance less than $d_a/10$, while their points are mutually at distance at least $d_a$, then recall how $\phi_1$ was defined on $P_a^{-1}(z)$.}
\end{proof}


\section{Open questions}\label{sec:open}

We recall that the polynomials in $\cal P[\theta]$ have degree $d$ and that
$I_{k}$ for $1\leq k\leq d-1$ denotes the subset of $\cal P[\theta]$ where there are exactly $k$ critical point in $\partial \Delta(P)$, counted with multiplicity.

For each bounded type rotation number:
\begin{enumerate}
  \item Is the maximum locus a single point? If not, what is the topology of this compact set?
  \item Is the maximum of $\rad(P)$ only reached on $I_{d-1}$?
\end{enumerate}

Recall $B(\theta)\in(\theta,+\infty]$ denotes the Bruno sum of $\theta\in\R$.
When $\theta$ varies in the set of bounded type numbers, $B(\theta)$ remains finite but can take arbitrarily big values.
We expect the typical value of $\rad(P)$, when $P\in\cal P[\theta]$ is in the connected locus, to be around $\exp(-B(\theta))$.
The most symmetric polynomial in $\cal P[\theta]$ is $P_0: z\mapsto z^d+e^{i2\pi \theta}z$, as it commutes with a rotation of order $d-1$ fixing $0$.
It is known (see \cite{BC1}) that $\rad(P_0)$ is sometimes small, yet much bigger than $\exp(-B(\theta))$.
\begin{enumerate}[resume]
  \item Is the maximum of $\rad(P)$ reached at $P_0$ when $\theta$ is the golden mean?
  \item For a more general bounded type $\theta$, is $\rad(P)$ locally maximal at $P_0$?
  \item Can one characterize the polynomial(s) for which the maximum is reached?
\end{enumerate}

There are polynomials with neutral fixed points with Bruno rotation numbers, and whose corresponding Siegel disk boundary does not contains any critical point.

\begin{enumerate}[resume]
  \item How to define the analogue of $I_{k}$ for a general Bruno rotation number?
  \item Is our main theorem still true for all Bruno rotation numbers?
\end{enumerate}

\bigskip

The following result is a direct consequence of \cite{BC1}, Proposition~A.3 page~346:
\begin{theorem}
  For every Bruno number $\theta$, the function $P\mapsto \log \rad(P)$ is pluri-super-harmonic in $\cal P[\theta]$.
\end{theorem}
This means that the restriction of $a \in\C^{d-2}\mapsto -\log \rad(P_a)$ to  every complex dimension one submanifold of $\C^{d-2}$ is subharmonic.
Equivalently, that in the sense of distributions, the $dd^c$ operator (a sort of pluri-Laplacian) applied to the function $a\mapsto -\log \rad(P_a)$ gives a positive current. It is a $(1,1)$-closed current, let us denote it $T$, which is a commonly used notation for this type of current.

Since $T$ is the $dd^c$ of a locally bounded (actually, continuous) plurisubharmonic function, its $n$-th exterior power is well-defined for all $n\geq 0$ (\cite{Klimek}, Section~3.4).

\begin{enumerate}[resume]
  \item\label{item:Q:support} For $1\leq k\leq d-2$, is the support of $T^{\wedge k}$ equal to $I_{k+1}$?
\end{enumerate}

In the case $d=3$, this has been proved in \cite{C:size}, Theorem~42.

\begin{lemma}\label{lem:rom}
Let $g$ be a continuous and plurisubharmonic function defined on a dimension $n$ complex manifold.
Assume that $g$ reaches its minimum and that the minimum locus $K$ is compact.
Then $K$ intersects the support of $(dd^c g)^{\wedge n}$.
\end{lemma}
\begin{proof}
If this were not the case, there would exist an open neighborhood $V$ of $K$ on which $(dd^c g)^{\wedge n}=0$.
We can arrange so that $V$ is compactly contained in the domain of $g$.
By Corollary~3.7.6 of \cite{Klimek} $g|_V$ is \emph{maximal} (see the definition at the beginning of Section~3.1 of \cite{Klimek}) and in particular satisfies the minimum principle (apply the definition of maximal functions with $u=g$ and $v=$ a constant).
Denote $m$ the minimum value of $g$.
Since $g\geq m+\eps$ on $\partial V$ for some $\eps>0$, the minimum principle implies that $g\geq m+\eps$ on $K$, leading to a contradiction.
\end{proof}
Applying this to $a\mapsto g(a)=-\log \rad(P_a)$, this gives us an alternate route to prove our main theorem:
\emph{if one can prove that the support of $dd^c T^{\wedge(d-2)}$ is contained\footnote{this is a weak form of the case $d-2$ of \eqref{item:Q:support}} in $I_{d-1}$, then our main theorem follows from \Cref{lem:rom}}.
This argument was communicated to us by Romain Dujardin.

\medskip

A laminar current is a current that is locally the integral of currents of integration over a measured family of disjoint disks, see \cite{BLS}.
\begin{enumerate}[resume]
  \item Is $T$ laminar?
\end{enumerate}

Motivated by the existing surgery from Blaschke fractions, and Zakeri's comments on the parameter space of such fractions (\cite{Zakeri-cubic}, Theorem~7.1), a daring conjecture would be the following:
\begin{conjecture}
Point \eqref{item:Q:support} holds.
The support of $T_\Lambda = dd^c (\lambda \in \Lambda \mapsto \log \rad(P_\lambda))=\Xi^{*}T$ is homeomorphic to the support of $dd^c f$ where $f(\lambda) = \min (|c_1|,\ldots,|c_{d-1}|)$ where we recall that $\lambda=(c_1,\ldots,c_{d-1})$ are the critical points of $P_\lambda$, whose product $c_1\cdots c_{d-1}$ is imposed by the rotation number.
As for $f$, the difference $\on{Supp} T_\Lambda - \on{Supp} T_\Lambda \wedge T_\Lambda =  I_{2}- I_{3}$ consists in a finite union over pairs of indices $i<j$ of critical points, of $d-2$-dimensional topological submanifolds foliated by a one real parameter family of $d-3$-dimensional complex submanifolds, the parameter being the relative angle between the two critical points on $\partial \Delta(P)$.
And $T_\Lambda$ is the sum over $i,j$ of the integral with respect to the Lebesgue measure on the relative angle, of the integration currents on these manifolds.
A similar description holds for $\on{Supp} T_\Lambda^{\wedge k} - \on{Supp} T_\Lambda^{\wedge (k+1)}$ and $T_\Lambda^{\wedge k}$, still by analogy with $f$.
In $J_{k} = \Xi^{-1}(I_k)$, the set of parameters $\lambda$ for which the relative internal angles between the critical point is fixed forms a complex submanifold of $\Lambda$, homeomorphic to $(\C-\ov\D)^{d-1-k}$.
\end{conjecture}


\section{Illustrations}\label{sec:illu}

\begin{figure}
  \centering
  \includegraphics[alt={The image shows a 3D rendered oblique downward view of the graph of the function $a\mapsto \log\rad(P_a)$, with level curves of equal variation textured mapped to the graph. It looks like the crest of a mountain surrounded by a vast plain with a gentle doward slope. There is a diffuse source of light making the crest more apparent, with one side of the mountain in the shade. The crest runs diagonally from the lower left to the upper right and takes about 2/3 of the image. The function is symmetric w.r.t\ $a\mapsto -a$, so the graph bears this symmetry. The crest is line, it looks fractal while the rest is very smooth. There are loca maxima of the function in the smooth part near the crest, visible where the level lines of the same value seem to take a tight turn. The position of the top of the mountain is hard to spot but it seems to be at, or close to, $a=0$. The position of the lowest points in the crest seem to be at its two ends.},width=\linewidth]{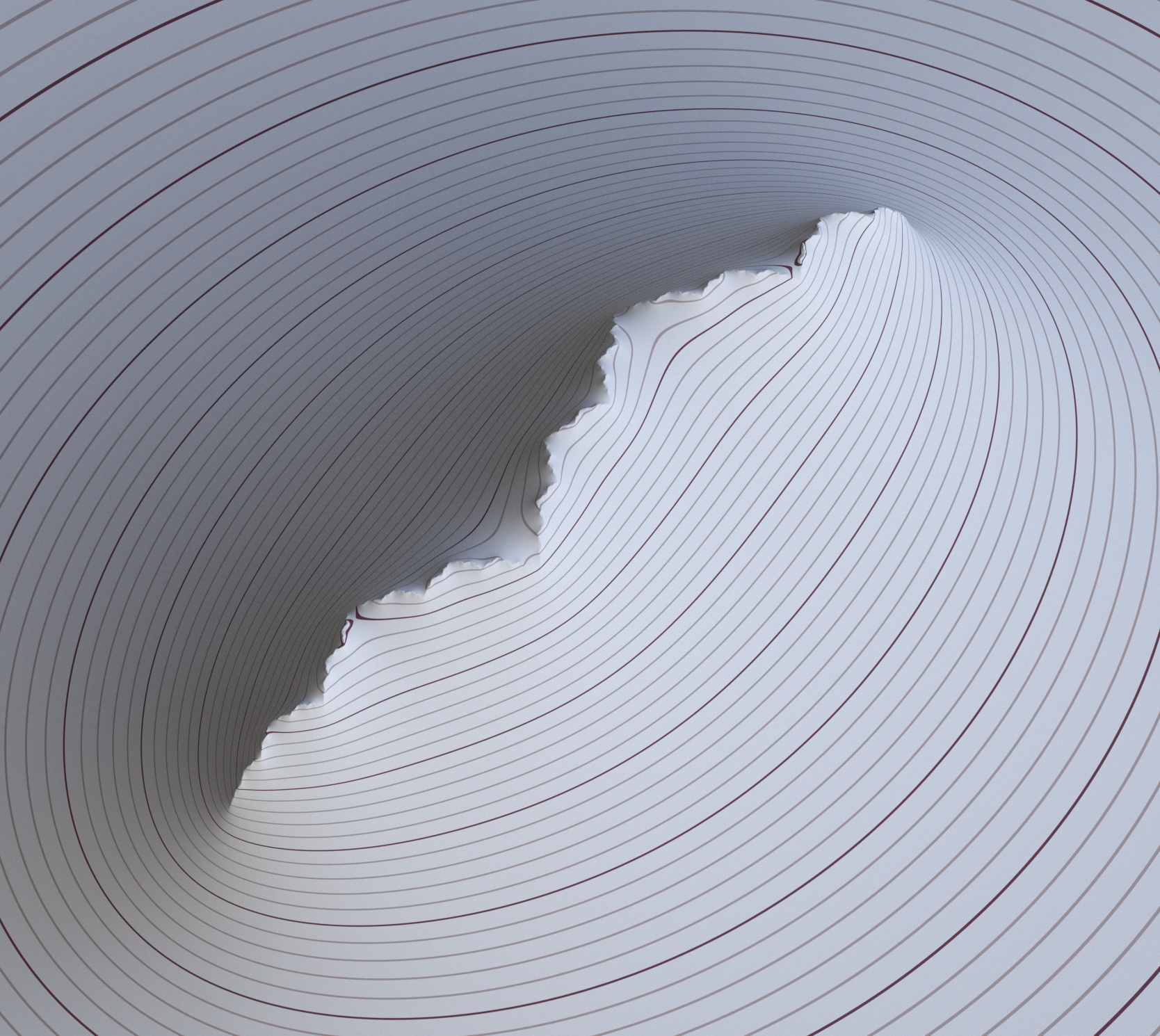}
  \caption{See the caption in \Cref{sec:illu}} 
  \label{fig:1}
\end{figure}

\begin{figure}
  \centering
  \includegraphics[alt={The image has been turned by 90 degrees. It is much higher than wide, by a factor slightly above 2. It shows an top view on the graph of \Cref{fig:1}, with the relief only barely guessable by a slight shading effect. The level lines texture have been replaced by a picture of the bifurcation locus of the family $a\mapsto P_a$, with a colour coding described in the main text. It is symmetric by $a\mapsto -a$. The crest and its fractal nature a is still prominent, with infinitely many green Siegel capture components of diameter tending to $0$ and densely attached to it. Two of them are at the two ends of the crest. Though they are known to be dense, only a few copies of the Mandelbrot set $M$ are big enough to be visible. One can spot tiny capture components attached to the top of their main antenna, and further tiny decorations too small to be identified. The two biggest $M$ copies ones are close to the top and bottom, with an infinite chain of capture components of smaller and smaller size tending to their main cusp and starting from the two components attached to the ends of the crest.},width=19.5cm,angle=90]{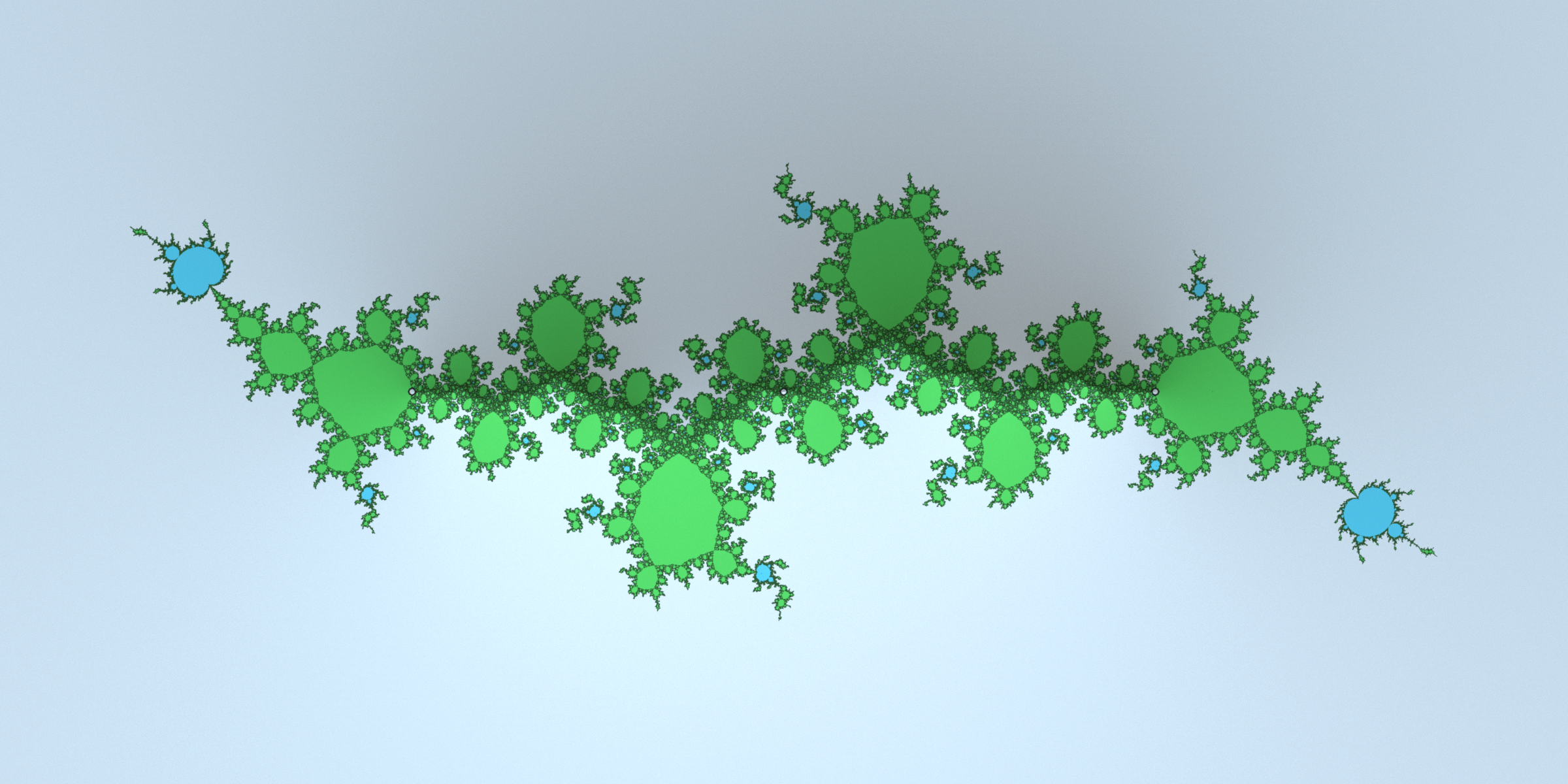}
  \caption{See the caption in \Cref{sec:illu}.}
  \label{fig:2}
\end{figure}

\Cref{fig:1} shows a 3D rendering of an oblique view on the graph of the function $a\in\C \mapsto \log \rad(P_a)$ with $\theta = \frac{\sqrt{5}-1}{2}$, $P_a = z^3+az^2+e^{2\pi i\theta} z$.
The conformal radius of $\Delta(P)$ is approximated by
\[\frac{1}{N}\sum_{n=0}^{N-1}\log |P^n(c)|\]
where $c$ is a point in $\partial \Delta(\theta)$ and here $N=1000$.
To guess which of $c_1$ or $c_2$ belongs $\partial \Delta(\theta)$ we used a hypothetical property: that for this value of $\theta$ and $d$, for all $a\in\C$, denoting $c$ a critical point in $\partial \Delta(P_a)$ and $v=P_a(c)$ the associated critical value, the Siegel disk contains a symmetric lens of vertices $0$ and $v$ opening angle some fixed $\alpha>0$ (we used $\alpha = 2\arctan(1/20)$).
Then a critical point whose orbit enters the lens attached to its own critical value cannot be on $\partial \Delta(P)$.
The actual algorithm involves some complications that we do not detail here.
Level curves are drawn, whose separation has been chosen for its visibility and informative character.
The maximum seems to be at, or near, the center of symmetry of the picture, which is at parameter $a=0$.
It looks like a mountain, whose crest is expected to be the support of $T$, i.e.\ the complement of the locus where $a\mapsto \log\rad(P_a)$ is locally harmonic.

\Cref{fig:2} shows a top view of \Cref{fig:1}, with the bifurcation locus of the family $a\mapsto P_a$ overlaid in dark green. The picture has been rotated by $90$ degrees. 
Three black dots mark $a=0$, for which $P_a$ commutes with $z\mapsto -z$ and the parameters $a=-\sqrt{3}$ and $a=\sqrt{3}$, for which the two critical points of $P_a$ coincide.
The cyan domains are the parameters for which there is an attracting cycle.
The lighter green domains are the parameters for which one of the critical points eventually falls in the Siegel disk (Siegel capture components).
It follows from \cite{Zakeri-cubic} that the set $I_{2}$ (parameters for which both critical points are on $\partial \Delta(P)$) is a Jordan arc.
It coincides with the support of $T$ according to \cite{C:size}, Theorem~42.

\printbibliography

\end{document}